\documentclass[11pt]{amsart}
\usepackage{geometry}                
\usepackage{graphicx}
\usepackage{amssymb}
\usepackage{epstopdf}
\title{Uniform Sobolev Inequality Along the Sasaki-Ricci Flow}
\author{Tristan C. Collins}

\begin{document}
\theoremstyle{plain}
\newtheorem{Lemma}{Lemma}[section]
\newtheorem{prop}{Proposition}[section]
\newtheorem{Theorem}{Theorem}[section]
\newtheorem{corollary}{Corollary}[section]
\newtheorem{definition}{Definition}[section]
\newtheorem{example}{Example}
\theoremstyle{remark}
\newtheorem*{remark}{Remark}
\newtheorem{thing}{Thing}

\maketitle
\begin{abstract}
In this paper we prove a uniform Sobolev inequality along the Sasaki-Ricci flow.  In the process, we develop the theory of basic Lebesgue and Sobolev function spaces, and prove some general results about the decomposition of the heat kernel for a class of elliptic operators on a Sasaki manifold.
\end{abstract}

\section{Introduction and Statement of Results}
There has recently been a great deal of interest in Sasaki geometry, particularly with regards to producing examples of Sasaki-Einstein metrics.  This interest has been fueled primarily by advances in the AdS/CFT correspondence in theoretical physics.  When the basic first Chern class is non-positive, the existence theory is well developed and generalizes the results in the K\"ahler setting.  However, when the basic first Chern class is positive there are known obstructions to existence; see for example the work of Futaki, Ono, and Wang \cite{Futaki}, and Gauntlett, Martelli, Sparks and Yau \cite{GMarSparYau}.  It is reasonable to expect that a suitably generalized version of the now famous conjecture of Yau \cite{Yau} should hold in the Sasaki setting;  that is, existence of Sasaki-Einstein metrics should be equivalent to some geometric invariant theory notion of stability.  We refer the reader to the survey article \cite{PSstab} and the references therein for an introduction to this very active area of research in the K\"ahler setting.  A great deal of ingenuity on the part of many different authors has produced a large number of quasi-regular Sasaki-Einstein manifolds; we refer the reader to the survey article \cite{Sparks} and the references therein.  Until recently there were no examples of Sasaki-Einstein metrics on irregular Sasaki manifolds;  in fact, such manifolds were conjectured to not exist by Cheeger and Tian \cite{CheeTian}.  The first examples of irregular Sasaki-Einstein manifolds were produced by Gauntlett, Martelli, Sparks and Waldrum in \cite{GMarSparW}, \cite{GMarSparW1}, and further developed by Martelli and Sparks \cite{MarSpar}, \cite{MarSpar1}.  New toric irregular Sasaki-Einstein metrics have been produced, for example on the circle bundle of a power of the anti-canonical bundle over the two point blow up of $\mathbb{C}P^{2}$, in the paper \cite{Futaki}.  These discoveries necessitate a more robust approach to the  existence theory for Sasaki-Einstein structures.  In response, a number of approaches have been developed by several different authors.  For example, Martelli, Sparks and Yau \cite{MarSparYau}, and Boyer, Galicki and Simanca \cite{BoyGalSim} developed an approach in the spirit of Calabi, by examining functionals related to the volume and scalar curvature.  A flow approach was developed by Smoczyk, Wang and Zhang in \cite{SmoWaZa}, where they introduced the Sasaki-Ricci flow, and generalized the results of Cao \cite{Cao} to Sasaki manifolds.  In a recent paper \cite{TristC}, the author generalized Perelman's functionals to the Sasaki setting and proved uniform bounds for the transverse scalar curvature and the transverse diameter.  In this paper we extend the analogy between the K\"ahler-Ricci flow and the Sasaki-Ricci flow by showing that these results  imply a uniform, basic Sobolev inequality along the Sasaki-Ricci flow.  Our main theorem is

\begin{Theorem}\label{main theorem}
Let $(S,g_{0})$ be a Sasaki manifold of dimension $m= 2n+1$, with basic first Chern class $c^{1}_{B}>0$.  Let $g(t)$ be a solution of the normalized Sasaki-Ricci flow on $[0,\infty)$ with $g(0)= g_{0}$, and let $d\mu_{t}$ be the volume form with respect to $g(t)$.  Then, for every $v \in W^{1,2}_{B}(S)$, we have
\begin{equation*}
\left(\int_{S} v^{2m/(m-2)}d\mu_{t}\right)^{(m-2)/m} \leq C \int_{S}\left( |\nabla v|^{2} +  v^{2}\right)d\mu_{t}
\end{equation*}
where $C$ depends only on $g_{0}$ and $m$.
\end{Theorem}
This result was obtained in the K\"ahler case in \cite{QZhang}, \cite{QZhangEr}, and subsequently developed further in \cite{Hsu}, \cite{RugYe}, and the book \cite{ZhangBk}.  However, the proof there does not carry over to the Sasaki setting.  Indeed, we have several difficulties to overcome, not the least of which is the absence of smooth, basic cutoff functions.  We overcome these difficulties by observing that the Reeb foliation imposes a great deal of symmetry on the heat kernel of any operator $L = \Delta - \Psi$, where $\Psi$ is a smooth, non-negative basic function.  Another difficulty we encounter is the absence of a well developed basic function theory, specifically with regards to basic Lebesgue and Sobolev spaces.  We take up this deficiency in the third section of this paper by defining these spaces, and examining some of their elementary properties.   Of course, the basic function theory on quasi-regular Sasaki manifolds follows immediately from the standard function theory on the orbifold quotient.  While this simpler case will be a guide for our work, we make a concerted effort to develop the theory without assuming any topological structure on the quotient, in order that it hold on irregular Sasaki manifolds.  Here, one must proceed with some care. The absence of basic partitions of unity supported in coordinate neighbourhoods, and particularly basic mollifiers presents a number of challenges, specifically with regards to regularization.  The basic observation is that the heat equation regularizes rough initial data, while also preserving the symmetry of the Reeb field.  The layout of this paper is as follows; in \S2 of this paper we provide a brief introduction to Sasaki geometry.  As we view this paper to be a direct extension of the results in \cite{TristC}, our introduction will be particularly minimal.  Instead, we refer the reader to \cite{BoyGal1}, \cite{Sparks} for a thorough introduction to Sasaki geometry.  In \S3 we develop the theory of basic Sobolev, and $L^{p}$ spaces, and describe the results from the standard theory which we will need.  In \S4 we discuss some properties of heat kernels on a Sasaki manifold, which will be essential for the proof.  In \S5 we prove Theorem~\ref{main theorem}.  The proof of the theorem has essentially three parts.  The first part of the argument consists of employing the monotonicity of the transverse $\mu$ functional proved in \cite{TristC} to obtain a uniform log Sobolev inequality along the flow.  In the second part of the proof, we use the log Sobolev inequality to obtain short time estimates for the heat kernel restricted to the basic functions.  In this part of the argument, the symmetry induced the the Reeb field is of particular importance.  Finally, we discuss how to obtain the uniform Sobolev inequality from the heat kernel estimates.  This final step is standard in the Riemannian case, but requires the developments from \S3 in the Sasaki case.  In the interest of brevity, and to avoid duplication, we will be somewhat terse in our presentation where the argument is a direct adaptation of the arguments in  \cite{QZhang},\cite{QZhangEr},\cite{RugYe}. 

{\bf Acknowledgements:}
I would like to thank my advisor Professor D.H. Phong for his guidance and encouragement, as well as for suggesting this problem.  I would also like to thank Professors Valentino Tosatti and  Gabor Szekelyhidi for many helpful suggestions during the writing of this paper. 

\section{Background}
The material in this section is taken directly from \cite{TristC}.  A Sasaki manifold of dimension $m=2n+1$ is a Riemannian manifold $(S^{2n+1},g)$ with the property that its metric cone $(C(S) = \mathbb{R}_{>0}\times S, \overline{g} = dr^2 + r^{2}g)$ is K\"ahler.  We will always assume that $S$ is closed and compact.  A Sasaki manifold inherits a number of key properties from its K\"ahler cone.  In particular, an important role is played by the Reeb vector field.  
\begin{definition}
The Reeb vector field is $\xi = J(r\partial_{r})$, where $J$ denotes the integrable complex structure on $C(S)$.  
\end{definition}
The restriction of $\xi$ to the slice $\{r=1\}$  is a unit length Killing vector field, and thus its orbits define a one-dimensional foliation of S by geodesics called the Reeb foliation.  Since the local flow of $\xi$ induces isometries of $(S,g)$, we are motivated to introduce the transverse distance function
\begin{equation*}
d^{T}(x,y) = \inf_{\{p\in orb_{\xi}x, q \in orb_{\xi}y\}} dist(p,q),
\end{equation*}
where $orb_{\xi}x$ is the orbit of $x$ under the group of isometries generated by the flow of $\xi$. 

\begin{definition}
A function $f \in C^{\infty}(S)$ is said to be \emph{basic} if $\mathcal{L}_{\xi}f=0$.  The set of smooth basic functions will be denoted by $C^{\infty}_{B}(S)$.
\end{definition}
Let $L_{\xi}$ be the line bundle spanned by the non-vanishing vector field $\xi$.  The contact subbundle $D\subset TS$ is defined as $D= \ker \eta$, where $\eta(X) = g(\xi,X)$ is a contact one form.  Throughout this paper we shall use the non-vanishing 2n+1-form $d\mu := \eta \wedge (d\eta)^{n}$ as the volume form on $S$.  We have the exact sequence
\begin{equation*}
0\rightarrow L_{\xi} \rightarrow TS \xrightarrow{p} Q \rightarrow 0.
\end{equation*}
The Sasaki metric $g$ gives an orthogonal splitting of this sequence $\sigma: Q\rightarrow TS$ so that we identify $Q \cong D$, and
\begin{equation*}
TS = D\oplus L_{\xi}
\end{equation*}
The quotient bundle $Q$ is endowed with a transverse metric $g^{T}$ by restricting $g$ to $D$ and using the splitting map $\sigma$ to identify $Q$ and $D$.  There is also a unique, torsion-free connection on $Q$ which is compatible with the metric $g^{T}$.  This connection is defined by
\begin{displaymath}
\nabla^{T}_{X}V = \left\{ \begin{array}{lr}
	\left(\nabla_{X}\sigma(V)\right)^{p},& \text{if } X \text{ is a section of D}\\
	\left[\xi, \sigma(V)\right]^{p},& \text{if } X=\xi
\end{array}
\right.
\end{displaymath}
where $\nabla$ is the Levi-Civita connection on $(S,g)$ and $p$ is the projection from $TS$ to the quotient $Q$.  This connection is called the \emph{transverse Levi-Civita connection} as it satisfies
\begin{equation*}
\nabla_{X}^{T}Y - \nabla^{T}_{Y}X - [X,Y]^{p}=0,
\end{equation*}
\begin{equation*}
Xg^{T}(V,W) = g^{T}(\nabla_{X}^{T}V,W)+ g^{T}(V, \nabla^{T}_{X}W).
\end{equation*}
In this way it is easy to see that the transverse Levi-Civita connection is the pullback of the Levi-Civita connection on the local Riemannian quotient.  We can then define the transverse curvature operator by
\begin{equation*}
Rm^{T}(X,Y) = \nabla_{X}^{T}\nabla_{Y}^{T} -\nabla_{Y}^{T}\nabla_{X}^{T}-\nabla_{[X,Y]}^{T},
\end{equation*}
and one can similarily define the transverse Ricci curvature, and the transverse scalar curvature.  A transverse metric $g^{T}$ is called a \emph{transverse Einstein metric} if it satisfies
\begin{equation*}
Ric^{T}=cg^{T}
\end{equation*}
for some constant $c$.  When the basic first Chern class is signed, say $c^{1}_{B}(S) = \kappa[g^{T}_{0}] $ we define the (normalized) Sasaki-Ricci flow by
\begin{equation}\label{SRF}
\frac{\partial g^T}{\partial t} = -Ric^{T} + \kappa g^{T}.
\end{equation}
This flow on transverse metrics was first studied in \cite{SmoWaZa} where the authors obtained long time existence for any $\kappa$, and exponential convergence in the case $\kappa \leq 0$.  On a foliated manifold, one can also consider the transverse Ricci flow
\begin{equation}\label{TRF}
\frac{\partial g^{T}}{\partial t}  = - Ric^{T},
\end{equation}
which was originally studied in \cite{Minoo} where they obtained short time existence.  On a Sasaki manifold, the transverse Ricci flow plays the analog of the unnormalized Ricci flow on a K\"ahler manifold.  Indeed, one can pass between a solution to~(\ref{SRF}) and~(\ref{TRF}) by scaling the metric and dilating time.  On a Sasaki manifold the analogies between the K\"ahler-Ricci flow and the flow~(\ref{SRF}) run deep.  For example, in \cite{TristC} (see also \cite{WHe} for another approach) the transverse $\mathcal{W}$ and transverse $\mu$ functionals were defined, and shown to be monotone along the flow, thereby opening Perelman's methods to the Sasaki setting.  Below, we define the transverse $\mu$ functional (from which one can infer the definition of the transverse $\mathcal{W}$ functional), and state a proposition concerning its monotonicity.  We refer the reader to~\cite{TristC} for more details concerning these functionals.  
\begin{definition}
Let $(S,g)$ be a Sasaki manifold, and $\tau>0$ a positive real number.  Set
\begin{equation*}
\chi = \left\{f\in C^{\infty}_{B} : \int_{S}(4\pi \tau)^{-n}e^{-f}d\mu=1\right\}.
\end{equation*}
Then define the functional $\mu^{T} $ by
\begin{equation*}
\mu^{T}(g,\tau) = \inf_{f \in \chi} (4\pi \tau)^{-n}\int_{S}\left( \tau(\text{R}^{T} + |\nabla f|^{2}) + (f-2n) \right) e^{-f}d\mu.
\end{equation*}
\end{definition}
\begin{prop}[\cite{TristC} Proposition 5.5]\label{mu monotone} 
Let $(S,g_{0})$ be a Sasaki manifold, $g(t)$ a solution of the transverse Ricci flow on $[0,T)$ with $g(0)=g_{0}$.  Assume that $\tau(t)$ is a positive function satisfying $\dot{\tau} =-1$.  Then $\mu^{T}(g(t), \tau(t))$ is increasing.  In particular, we have $\mu^{T}(g(t),r^{2}) \geq \mu^{T}(g(0),t+r^{2})$.
\end{prop}
These functionals were put to use in~\cite{TristC} to generalize a major result of Perelman on the K\"ahler-Ricci flow.  The precise result is
\begin{Theorem}[\cite{TristC} Theorem 1.3]\label{Pman thm}
Let $g^{T}(t)$ be a Sasaki-Ricci flow on a compact Sasaki manifold $(S,\xi)$ of real dimension $2n+1$, and transverse complex dimension n, with $c^{1}_{B}(S)>0$.  Let $u \in C^{\infty}_{B}(S)$ be the transverse Ricci potential.  Then there exists a uniform constant $C$, depending only on the initial metric $g(0)$ so that
\begin{equation}
|R^{T}(g(t))| + |u|_{C^{1}} + diam^{T}(S,g(t)) < C
\end{equation}
where $diam^{T}(S,g(t)) = \sup_{x,y\in S} d^{T}(x,y)$.
\end{Theorem} 

\section{Basic $L^{p}$ and Sobolev Spaces}\label{basic sob spaces}
We begin this section somewhat anachronistically, by introducing the basic Sobolev spaces.  We will provide two alternative definitions of these spaces, and prove they are equivalent.  As mentioned in the introduction, one difficulty is the absence of partitions of unity on Sasaki manifolds, and the lack of a suitable mollifier.  It is easy to see, for example, that basic mollifiers cannot exist when the Sasaki structure is irregular.  Instead, we shall use the symmetry of the manifold under the action induced by the Reeb field to deduce particularly useful properties of the heat kernel.  We remark that the Sobolev theory is somewhat easier to develop than the $L^{p}$ theory, as the presence of at least one weak derivative provides a particularly simple characterization of the ``weakly" basic functions.
\begin{definition}
Define the space $W^{r,2}_{B}$ to be the closure of $C^{\infty}_{B}(S)$ in the norm induced by $W^{r,2}(S)$. 
\end{definition}
As promised, we now provide an alternative definition of the basic Sobolev spaces;

\begin{Lemma}
We have, $ H = \left\{ f \in W^{r,2}(S) : \int_{S} f (\xi \phi) d\mu = 0 \text{  }\text{ for all } \phi \in C^{\infty}(S) \right\} = W^{r,2}_{B}$.
\end{Lemma}

\begin{proof}
Given $f \in H$, let $f(x,t)$ be the solution to the heat equation on $S$, with $f(x,0) =f$.  Then we can write
\begin{equation*}
f(x,t) = \int_{S} P(x,y,t)f(y)d\mu(y).
\end{equation*}
By the definition of $H$, and the fact that $P(x,y,t)$ is a smooth function for $t>0$ with $P(x,y,t) = P(y,x,t)$, we see immediately that $f(x,t)$ is basic for $t>0$.  Moreover, $f(x,t)$ converges to $f$ in $W^{r,2}$ as $t \rightarrow 0^{+}$.  Thus $H \subset W^{r,2}_{B}$.  The containment $C^{\infty}_{B} \subset H$ is clear.  We need only show that $H$ is closed.  Suppose $\{f_{j}\}$ is a Cauchy sequence in $H$.  Then there exists $f \in W^{r,2}$, and a subsequence (not relabeled), such that $f_{j}$ converges to $f$ in $W^{r,2}$ and $f_{j} \rightarrow f$ pointwise a.e.  Fix $\phi \in C^{\infty}$, and let $\psi = \xi \phi$.  Then $f_{j}\psi \rightarrow f\psi$ a.e., $f_{j}\psi \in L^{1}$, and $|f_{j} \psi| \leq |f_{j}|^{2} + |\psi|^{2}$.
Thus, the dominated convergence theorem yields $0 = \lim_{j} \int_{S} f_{j}(\xi \phi)  = \int_{S} f (\xi \phi)$.
Thus, $f \in H$, and hence $H$ is closed.  This proves the lemma.
\end{proof}


We now proceed to the basic $L^{p}$ spaces.  In a similar fashion as above, we define:

\begin{definition}
Define the space $L^{p}_{B}$ to be the closure of $C^{\infty}_{B}$ in the norm induced by $L^{p}$.
\end{definition}

We now give an alternative description of the spaces $L^{p}_{B}$, which is decidedly less attractive than the characterization of the basic Sobolev spaces due to the absence of a weak derivative. 

\begin{definition}
We say that a measurable function $f$ is basic if there exists a measurable function $g$ such that for any point $x$, $g(x) = g(y)$ for every $y \in orb_{\xi}x$, and $f = g$ a.e.  We say that a Borel set $U \subset S$ is a basic set if the indicator function for $U$, $\chi_{U}$, is a basic function.
\end{definition}
\begin{remark}
We note that in the above description it is clear that a basic measurable function $f$ is almost everywhere equal to a measurable function $g$ with the property that $g$ has an everywhere defined derivative in the $\xi$ direction, and $\xi g \equiv 0$.  This is, of course, not an accident.
\end{remark}

\begin{Lemma}
Let $H = \{ f \in L^{p} : f \text{ is basic } \}$.  Then $H = L^{p}_{B}$.
\end{Lemma}

\begin{proof}
We begin by showing that $L^{P}_{B} \subset H$.  We clearly have $C^{\infty}_{B} \subset H$, and so it suffices to show that $H$ is a closed subspace of $L^{p}$.  Let $\{f_{j} \}$ be a Cauchy sequence in $H$, and without loss of generality, assume that for every $j$ and every $x \in S$, $f_{j}(x) = f_{j}(y)$ is verified for every $y\in orb_{\xi}x$.  Let $f$ be the limit point of the $f_{j}$'s in $L^{p}$.  By passing to a subsequence we may assume that $f_{j} \rightarrow f$ a.e.  Let $x \in S$ be a point where $\lim_{j} f_{j}(x) = f(x)$.  Since the sequence $f_{j}(x)$ converges, the sequence $f_{j}(y)$ converges for any $y\in orb_{\xi}x$, and we may take $f(y)$ to be given by this limit.  Thus, for almost every $x \in S$, we have that $f(x) = f(y)$ for every $y\in orb_{\xi} x$.  By redefining $f$ on a set of measure zero, we see that $f \in H$.  To show that $H \subset L^{p}_{B}$, it suffices to show that $\chi_{E} \in L^{p}_{B}$ for every basic Borel set $E$.  In fact, we may assume that $E$ is closed, and that if $x \in E$, then $\overline{orb_{\xi}x} \subset E$, as these modifications are measure zero.  For each $0 < \epsilon < 1$, define the function $f_{\epsilon}$ by
\begin{displaymath}
f_{\epsilon}(x) = \left\{ \begin{array}{lr}
	\epsilon^{-1}d^{T}(x,E),&\text{if } d^{T}(x,E) \leq \epsilon \\
	1,& \text{ if } d^{T}(x,E)>\epsilon.
\end{array}
\right.
\end{displaymath}
Note that $f_{\epsilon}=0$ if and only is $x \in E$, since $E$ is closed and basic.  Moreover, we clearly have $f_{\epsilon}$ converging to $1-\chi_{E}$ in $L^{p}$ as $\epsilon \rightarrow 0$.  It is easy to check that  $|f_{\epsilon}(x)-f_{\epsilon}(y)| \leq \epsilon^{-1} d(x,y)$.  It follows by Rademacher's theorem that $f_{\epsilon}$ is a.e. differentiable, and $|\nabla f_{\epsilon}| \leq \epsilon^{-1}$.  It is clear that $\xi f_{\epsilon} =0$ a.e.  Thus, $f_{\epsilon} \in W^{1,2}_{B}(S)$, and so there exists a sequence of functions $g_{\epsilon, j} \in C^{\infty}_{B}$ with $g_{\epsilon, j} \rightarrow f_{\epsilon}$ in $L^{2}$ as $j \rightarrow \infty$.  The $g_{\epsilon, j}$ are obtained via heat equation regularization, and so by the maximum principle $|g_{\epsilon, j} | \leq 1$.  By passing to a subsequence, we may assume that $g_{\epsilon, j}$ converges to $f_{\epsilon}$ pointwise almost everywhere and by the dominated convergence theorem, $g_{\epsilon, j} \rightarrow f_{\epsilon}$ in $L^{p}$.  By passing to a diagonal subsequence, we then obtain the existence of a sequence of smooth, basic functions converging to $\chi_{E}$ in $L^{p}$.  Since the indicator functions for basic sets are clearly dense in $H$, we are done.
\end{proof}

We now note a few corollaries of the above proof, which show that $L^{p}_{B}$ has many properties in common with the usual $L^{p}$ space;

\begin{corollary}
If $f \in L^{p}_{B}$ and $U\subset \mathbb{R}$ is any open set, then $f^{-1}(U)$ is a basic subset of $S$.
\end{corollary}

\begin{corollary}\label{simple basic dense}
The basic simple functions, $\Sigma_{B}$, are dense in $L^{p}_{B}$ for every $1\leq p < \infty$.
\end{corollary}

\begin{corollary}
Let p,q be conjugate exponents.  Suppose that $g \in L^{q}_{B}(S)$.  Then
\begin{equation*}
\|g\|_{q} = \sup \left\{ \left|\int_{S} fg d\mu \right| : f\in \Sigma_{B} \text{ and } \|f\|_{p} =1 \right\}.
\end{equation*}
\end{corollary}

The above lemmas allow us to obtain basic versions of the Riesz-Thorin, and the Marcinkiewicz interpolation theorems. The proofs of these theorems are elementary generalizations of the standard proofs.  The proofs can be found in any standard book on measure theory, for example \cite{Folland}.  In particular, we have:

\begin{Theorem}[Riesz-Thorin interpolation theorem]
Let $(S,g)$ be a Sasaki manifold with measure induced by the standard volume form.  Suppose that $p_{0}, p_{1}, q_{0}, q_{1} \in [1,\infty]$.  For $0<t<1$, define $p_{t}$ and $q_{t}$ by
\begin{equation*}
\frac{1}{p_{t}} = \frac{1-t}{p_{0}} + \frac{t}{p_{1}} \text{ , }\text{   } \frac{1}{q_{t}} = \frac{1-t}{q_{0}} + \frac{t}{q_{1}}.
\end{equation*}
If $T$ is a linear map from $L^{p_{0}}_{B} + L^{p_{1}}_{B}$ into $L^{q_{0}}_{B} + L^{q_{1}}_{B}$ such that $\|Tf\|_{q_{0}} \leq M_{0}\|f\|_{p_{0}}$ for $f \in L^{p_{0}}_{B}$ and $\|Tf\|q_{1} \leq M_{1}\|f\|_{p_{1}}$ for $f \in L^{p_{1}}_{B}$, then $\|Tf\|_{q_{t}} \leq M_{0}^{1-t}M_{1}^{t} \|f \|_{p_{t}}$, for each $0< t < 1$.
\end{Theorem}

\begin{Theorem}[Marcinkiewicz interpolation theorem]
Let $(S,g)$ be a Sasaki manifold, with measure induced by the standard volume form.  Suppose that 
$p_{0}, p_{1}, q_{0}, q_{1} \in [1,\infty]$, such that $p_{0} \leq q_{0}$, $p_{1} \leq q_{1}$, and $q_{0} \ne q_{1}$.For $0<t<1$, define $p$ and $q$ by
\begin{equation*}
\frac{1}{p} = \frac{1-t}{p_{0}} + \frac{t}{p_{1}} \text{ , }\text{   } \frac{1}{q} = \frac{1-t}{q_{0}} + \frac{t}{q_{1}}.
\end{equation*}
If $T$ is a sublinear map from $L^{p_{0}}_{B}$ to the space of basic measurable functions on S, which is both weak type $(p_{0},q_{0})$ and weak type $(p_{1},q_{1})$, then T is strong type $(p,q)$.  That is, if $[Tf]_{q_{j}} \leq C_{j} \|f\|_{p_{j}}$ for $j=0,1$, then $\|Tf\|_{q} \leq B_{p}\|f\|_{p}$ where $B_{p}$ depends only on $p_{j}, q_{j}, C_{j}, p$.  Moreover, for $j=0,1$, $B_{p}|p-p_{j}|$ (respectively $B_{p}$) remains bounded as $p\rightarrow p_{j}$ if $p_{j} <\infty$ (respectively $p_{j} =\infty$).
\end{Theorem}

\section{The Basic Heat Kernel}\label{basic heat kernel section}
In this section we examine the heat kernel on $L^{2}$ associated to any elliptic operator $L = \Delta - \Psi$ for  $\Psi$ a positive, smooth, basic potential function.  We will show that the restriction of the heat kernel to $L^{2}_{B}$ is defined by a smooth, basic function for $t>0$.  Our construction depends strongly on the following lemmas, which we apply to the \emph{positive} operator $-L$.

\begin{Lemma}[\cite{KamTond} Proposition 4.5]\label{basic Rellich}
$\forall r\geq 0$ and $t>0$ the inculsion $H^{r+t}_{B}(S) \hookrightarrow H^{r}_{B}(S)$ is compact.
\end{Lemma}

\begin{Lemma}\label{basic elliptic reg}
Let $L$ be an elliptic operator, and suppose that $\phi \in H^{r+1}_{B}$ has $L\phi \in H^{r}_{B}$.  Then $\phi \in H^{r+2}_{B}$, and
\begin{equation*}
\| \phi \|_{r+2, B} \leq C_{r} \left( \| L\phi \|_{r,B} + \| \phi \|_{r+1,B} \right).
\end{equation*}
\end{Lemma}
It was observed in \cite{TristC} that for such an operator, Lemmas~\ref{basic Rellich} and~\ref{basic elliptic reg} yield the existence of a complete, orthonormal set $\{\phi_{i}, \lambda_{i} \}$ of eigenfunctions and eigenvalues for $-L$, spanning $L^{2}_{B}$, such that each $\phi_{i}$ is smooth and basic.  We can then define the basic heat kernel of the operator $L-\partial_{t}$ to be
\begin{equation}\label{basic heat kernel}
P^{B}(x,y,t) = \sum_{i} e^{-\lambda_{i}t}\phi_{i}(x)\phi_{i}(y).
\end{equation}
By the usual elliptic theory, the set $\{ \phi_{i}, \lambda_{i} \}$ can be completed into a complete set of eigenfunctions and eigenvlaues for the space $L^{2}$.  We may choose eigenfunctions $\{ \psi_{i}, \sigma_{i} \}$ so that $\{\phi_{i}, \lambda_{i} \} \cup \{ \psi_{i}, \sigma_{i} \}$ span $L^{2}$, and $\psi_{j} \perp \phi_{i}$ for each pair $i,j$ with respect to the $L^{2}$ inner-product.  The heat kernel for the parabolic operator $L-\partial_{t}$ is then given by
 \begin{equation}\label{heat kernel}
 P(x,y,t) = P^{B}(x,y,t) + P^{\perp}(x,y,t) = \sum_{i} e^{-\lambda_{i}t}\phi_{i}(x)\phi_{i}(y) + \sum_{i} e^{-\sigma_{i}t}\psi_{i}(x)\psi_{i}(y).
\end{equation}
We remark that the spectral representation in~(\ref{heat kernel}) is generally formal.  However, in the present case, the defining sum converges uniformly to a smooth function on $S \times S \times (0, \infty]$.  This follows from the following lemmas, whose proofs we only sketch.  

\begin{Lemma}\label{elliptic eigen est}
Suppose $L\phi = -\lambda \phi$, and $\| \phi \|_{L^{2}} =1$.  Then $\|\phi\|_{\infty} < C (|\lambda|+1)^{m_{n}}$ for constants $C$ and $m$ depending only on the dimension of $S$ and the Sobolev imbedding theorem.  Moreover, $\lambda \geq 0$, with equality if and only if $\phi \equiv 0$.
\end{Lemma}
\begin{proof}
Apply the elliptic $W^{p,2}$ estimate inductively, until $p$ is sufficiently large to apply the Sobolev imbedding theorem.  The second statement follows from the positivity of $\Psi$, and integration by parts.
\end{proof}

\begin{Lemma}
The defining sum~(\ref{heat kernel}) for the heat kernel associated to the operator $L$ converges uniformly to a smooth function for each time $t>0$.  In particular, the basic heat kernel~(\ref{basic heat kernel}) is a smooth, basic function.
\end{Lemma}
\begin{proof}
Fix a time $T>0$.  We apply the bound in Lemma~\ref{elliptic eigen est}, and H\"ormander's estimate \cite{LarsH} for the spectral counting function to show that the spectral representations for $P, P^{B}$ and $P^{\perp}$ converge uniformly for every $t>T$ to a continuous function.  We now employ Bernstein's principle to extend the $C^{0}$ bound, $\|\phi\|_{\infty} < C (|\lambda|+1)^{m_{n}}$ in Lemma~\ref{elliptic eigen est}, to a $C^{k}$ bound for $\| \phi \|_{C^{k}} \leq C h(|\lambda|)$ for a constant $C$ independent of $\lambda$, and a polynomial $h$ of order $m_{k}$ independent of $\lambda$.  Again, H\"ormander's estimate  for the spectral counting function yields convergence.
\end{proof}  

From the defining formula~(\ref{heat kernel}), we see that, for any $f \in L^{2}_{B}$ we have $\int_{S}P^{\perp}(x,y,t)f(y)d\mu(y) = 0$.  That is, the restriction of $P$ to $L^{2}_{B}$ is $P^{B}$. One obtains still more information about the integral kernel $P^{B}$ by applying the maximum principle to the equation $L-\partial_{t} =0$. 
\begin{Lemma}\label{pos contract}
If $\phi$ solves $L\phi-\partial_{t}\phi=0$, with $\phi(x,t)=\phi(x,0)$ smooth, then $\max_{x}\phi^{+}(x,t)$, and $\max_{x}\phi^{-}(x,t)$ are strictly decreasing.  If $\phi(0)>0$, then $\phi(x,t) >0$ for all $t$.  
\end{Lemma}
By the density of $C^{\infty}_{B}$ in the basic Sobolev and Lebesgue spaces, we see that the integral kernel $P^{B}(x,y,t)$ is a positivity-preserving contraction.  Now, since the function $f \equiv 1$ is trivially basic, we have $1 = \int_{S} P(x,y,0)d\mu(y) = \int_{S} P^{B}(x,y,0)d\mu(y)$.  Thus, for any $t>0$,
\begin{equation*}
\frac{\partial}{\partial t} \int_{S} P(x,y,t) d\mu(y) = \frac{\partial}{\partial t} \int_{S} P^{B}(x,y,t) d\mu(y) = -\int_{S}P^{B}(x,y,t)\Psi(y)d\mu(y) <0
\end{equation*}
Where the last inequality uses that $\Psi>0$, and $P^{B}$ is positivity preserving.  We thus see $\int_{S}P^{B}(x,y,t)d\mu(y)\leq 1$.  


\section{Proof of Theorem~\ref{main theorem}}
Throughout this section we will assume, without loss of generality, that $\kappa = 1$, so that the transverse Ricci flow exists on $[0,1)$.
\subsection{Restricted log Sobolev Inequality}
We begin by proving a restricted log Sobolev inequality, which will allow us to obtain a uniform upper bound for $P^{B}$.  

\begin{prop}\label{restricted log Sob}
Let $g(t)$ be a solution of the Sasaki-Ricci flow defined on $[0,\infty)$, with $g(0)$ Sasaki.  For any $0 < \epsilon \leq 2$, $t\in [0,\infty)$ and any $v\in W^{1,2}_{B}$ with $\|v\|_{L^{2}(g(t))} =1$, we have
\begin{equation*}
\int_{S} v^{2}\log v^{2} d\mu_{t} \leq \epsilon^{2}\int_{S}\left(|\nabla v|^{2} + \frac{1}{4}R^{T}v^{2}\right)d\mu_{t} -(2n+1)\log \epsilon + C_{1} + \max_{x\in S} R_{-}^{T}(x,0)
\end{equation*}
for a constant $C_{1}$ depending only on $n$ and $g(0)$.
\end{prop}

The proof of this theorem relies on the monotonicity of $\mu^{T}$ along the transverse Ricci flow on a Sasaki manifold, and on previously known log Sobolev inequalities.  Before beginning the proof of Proposition~\ref{restricted log Sob} we state the precise version of the log Sobolev inequality we will need, which is well known; see, for instance, \cite{RugYe} Theorem 3.3 for a proof.

\begin{prop}\label{input log Sob}
Let $(M^{n}, g)$ be a Riemannian manifold, and $\lambda \in (0, 2]$.  For any $v \in W^{1,2}(M)$ with $\|v\|_{2} =1$, there exists a positive constant $C_{0}$ depending only on $g$ such that
\begin{equation*}
\int_{S}v^{2}\log v^{2} d\mu \leq \lambda^{2} \int_{S} |\nabla v|^{2} d\mu-n\log \lambda + C_{0}.
\end{equation*}
\end{prop}

\begin{proof}[Proof of Proposition~\ref{restricted log Sob}]
Fix $\epsilon \in (0,2]$.  Recall that if $g^{T}(s)$ is a solution of the Sasaki-Ricci flow on $[0,\infty)$, then $\tilde{g}^{T}(t)= (1-t)g^{T}(-\log(1-t))$ is a solution of the transverse Ricci flow on $[0,1)$.  Fix a time $s\in[0, \infty)$, and define $t$ by $s=-\log(1-t)$.  By Proposition~\ref{mu monotone}, we have (upon suppressing the superscript $T$)
\begin{equation}\label{mu monotone eqn}
\mu(g(s),\epsilon^{2}) = \mu(g(-\log(1-t)), \epsilon^{2}) = \mu(\tilde{g}(t),(1-t)\epsilon^{2}) \geq \mu(g(0), (1-t)\epsilon^{2}+t).
\end{equation}

Let $\sigma(t) = t+ (1-t)\epsilon^{2}$, and observe that $\sigma(t) \in [\epsilon^{2}, 1]$;  combining this with equation~(\ref{mu monotone eqn}), and Proposition~\ref{input log Sob} we obtain
\begin{equation*}
\begin{aligned}
&\inf_{\left\{v \in W^{1,2}_{B} : \|v\|_{L^{2}(g(s))} =1\right\}} \int_{S} \epsilon^{2}\left(R^{T}_{g(t)}v^{2} +4|\nabla v|_{g(t)}^{2}\right) -v^{2} \log v^{2} d\mu_{t} -n\log(\epsilon^{2})  \geq \\ 
&\inf_{\left\{v_{0} \in W^{1,2}_{B} : \|v_{0}\|_{L^{2}(g(0))} =1\right\}}\int_{S}\sigma(t) \left( R^{T}_{g(0)}v_{0}^{2} +4|\nabla v_{0}|_{g(0)}\right) -v_{0}^{2} \log v_{0}^{2} d\mu_{0}-n\log( \sigma(t)) \geq\\ 
&-\sup_{x\in S} R^{T}_{-}(x,0) +\frac{1}{2} \log(\sigma(t))-C_{1}.
\end{aligned}
\end{equation*}
where $R^{T} = R^{T}_{+} - R^{T}_{-}$, and we have used $\lambda = \sqrt{\sigma(t)}$ in our application of Proposition~\ref{input log Sob}.  Note that $C_{1}$ depends only on $g(0)$ and $n$.  Rearranging this equation, and using that $\sigma(t) \in[\epsilon^{2}, 1]$ yields,
\begin{equation*}
\int_{S} v^{2}\log v^{2} d\mu_{t} \leq \epsilon^{2}\int_{S}\left(4|\nabla v|^{2} + R^{T}v^{2}\right)d\mu_{t} -(2n+1)\log \epsilon + C_{1} + \max_{x\in S} R_{-}^{T}(x,0).
\end{equation*}
Redefining $\epsilon$, we are done.
\end{proof}
\subsection{Uniform Upper Bound for the Basic Heat Kernel}
Fix a time $t_{0}$ in the Sasaki-Ricci flow.  We now use the Proposition~\ref{restricted log Sob} to obtain a uniform upper bound for the operator
\begin{equation}\label{Theta heat equation}
\Delta_{g(t_{0})} -\frac{1}{4}R^{T}(x,t_{0}) -\Theta -1,
\end{equation}
where $\Theta=\sup_{x \in S} R^{T}_{-}(x,0)$.  It is standard to compute that $ \sup R^{T}_{-}(x,t)$ is increasing under the flow, and hence $\frac{1}{4}R^{T}(x,t) +\Theta >0 $ for all $t>0$.  Let $P_{\Theta}$ be the heat kernel of the operator~(\ref{Theta heat equation}).  By the discussion in \S\ref{basic heat kernel section} we may decompose this operator as $P_{\Theta} = P_{\Theta}^{B} + P_{\Theta}^{\perp}$.  Let $u(x,t)$ be a positive, basic solution of the equation
\begin{equation*}\label{curvature heat equation}
\Delta_{g(t_{0})}u(x,t) -(\frac{1}{4}R^{T}(x,t_{0}) +\Theta +1)u(x,t) = 0.
\end{equation*}
Note that any solution to this equation with smooth, positive, basic initial data will be positive and basic for all time by Lemma~\ref{pos contract}.  For simplicity, we now suppress the argument $t_{0}$, with the understanding that it is fixed, and all computations will be done with respect to the \emph{fixed} metric $g(t_{0})$;  we will denote by $\Psi(x)$ .
We now follow the computation in \cite{QZhang}, which carries over verbatim for the basic function $u(x,t)$.  For completeness, and as our operator is a slight modification of the operator in \cite{QZhang} we include a few steps in the computation.  Choose $T \in (0,1]$ and set $p(t) = T/(T-t)$, so that $p(0)=1$, $p(T)=\infty$.  For convenience we define the function $v(x,t) = u^{p/2}/\|u^{p/2}\|_{L^{2}}$, which has the property that $v \in W^{1,2}_{B}(dx)$ and $\|v\|_{L^{2}(dx)} =1$.  One then computes
\begin{equation}\label{norm derivative}
\begin{aligned}
&p^{2}(t)\partial_{t}\log \|u\|_{p(t)}\\ &= p'(t) \int_{S}v^{2}\log v^{2}dx - (4p-1)\int_{S}|\nabla v|^{2} dx-p^{2}\int_{S}\Psi(x) v^{2}dx \\
&= p'\left\{\int_{S}v^{2}\log v^{2} -\frac{4(p-1)}{p'}\left(\int_{S} |\nabla v|^{2} +\Psi(x) v^{2} dx \right)\right\} + \left((4p-1)-p^{2}\right)\int_{S}\Psi v^{2}dx.
\end{aligned}
\end{equation}
It is easy to check that $4(p-1)/p' \leq T\leq 1$, and $-T \leq \left(4(p-1)-p^{2}\right)/p' \leq 0$.  Since $\Psi>0$, the final terms in~(\ref{norm derivative}) is negative.  We apply Proposition~\ref{restricted log Sob} with $\epsilon = 4(p-1)/p' $  to obtain
\begin{equation*}
\partial_{t}\log \|u\|_{p(t)} \leq \frac{1}{T}\left( -\frac{2n+1}{2}\log\left( \frac{4t(T-t)}{T}\right) +C_{1} \right)
\end{equation*}
Integrating in $t$ from $0$ to $T$ and rearranging yields that for every smooth function $u \in C^{\infty}_{B}$, and every $0<T \leq 1$ we have
\begin{equation*}
|u(x,T)| = \left|\int_{S} P_{\Theta}(x,y,T)u(y)dy\right| =\left| \int_{S} P_{\Theta}^{B}(x,y,T) u(y) dy \right|\leq T^{-(2n+1)/2}\Lambda \|u \|_{L^{1}(dx)}.
\end{equation*}
Since $P^{B}_{\Theta}(x,y,t)$ is basic, it follows that for every $T \in (0,1]$ we have the bound
\begin{equation}\label{short time upper bound}
P^{B}_{\Theta}(x,y,T) \leq CT^{-(2n+1)/2}.
\end{equation}
where $C$ is a constant that depends only on $g(0)$ and $n$.

\subsection{Uniform Sobolev Inequality}
We begin by showing that the short time upper bound~(\ref{short time upper bound}) extends to a bound for every positive time $t$.  Fix $y \in S$ and consider the function $f(x,t) = P_{\Theta}^{B}(x,y,1+t)$, which is a smooth, basic function.  Equation~(\ref{short time upper bound}) shows the $f(x,0) \leq C$, and we have
\begin{equation*}
\frac{\partial}{\partial t} f(x,t) = \Delta f(x,t) - \Psi(x) f(x,t),
\end{equation*}
where $\Psi \geq 1$.  We then compute $ \partial_{t} \left(e^{t}f(x,t)\right)$ and apply the maximum principle to see that $f(x,t) \leq e^{-t}C$, whence 
\begin{equation}\label{lt upper bound}
P^{B}_{\Theta}(x,y,t) \leq CT^{-(2n+1)/2}, \text{  }\text{  }\text{ for every } t > 0.
\end{equation}
By the discussion following Lemma~\ref{pos contract}, $P_{\Theta}$ is a contraction on $L^{1}$, and so $P_{\Theta}^{B}$ is a contraction on $L^{1}_{B}$.  H\"older's inequality, combined with the long time upper bound~(\ref{lt upper bound}) yields that, for any $f \in L^{2}_{B}$, we have
\begin{equation}\label{semi group bound}
\left| \int_{S} P_{\Theta}^{B}(x,y,t)f(y)d\mu(y) \right| \leq \Lambda^{1/2} t^{-(2n+1)/4} \|f\|_{2}.
\end{equation}
In the Riemannian case, passing from a bound of the type~(\ref{semi group bound}) to the Sobolev inequality is by now a standard argument; we refer the reader to \cite{Davies}, \cite{RugYe}.  In our case, the argument is similar, with the added complication that the bound~(\ref{semi group bound}) only holds on $L^{2}_{B}$, which is a closed subspace of $L^{2}$.  However, since the operator defined by the integral kernel preserves the basic spaces, we can restrict our attention to the \emph{basic} Lebesgue and Sobolev spaces.  The ingredients in the proof which don't immediately carry over are none other than the Riesz-Thorin interpolation theorem and the Marcinkiewicz interpolation theorem, which we established in \S\ref{basic sob spaces}.  With these considerations, we do not feel it is irresponsible to omit the details.  We thus obtain the existence of constants $A, B > 0$ depending only on $n$ and $g_{0}$ such that, for each time $t$ in the Sasaki-Ricci flow, we have
\begin{equation*}
\left(\int_{S} v^{2m/(m-2)}d\mu_{t}\right)^{(m-2)/m} \leq A \int_{S}\left( |\nabla v|^{2} + \frac{1}{4}R^{T}v^{2}\right)d\mu_{t} + B\int_{S} v^{2}d\mu_{t}.
\end{equation*}
By applying Theorem~\ref{Pman thm} to bound $R^{T}$ above uniformly, we establish Theorem~\ref{main theorem}.

\end{document}